\documentclass[12pt]{article}
\usepackage{amssymb}
\usepackage{amsfonts}
\usepackage{amsmath}
\usepackage[usenames]{color}
\usepackage{mathrsfs}
\usepackage{amsfonts}
\usepackage{amssymb,amsmath}
\usepackage{CJK}
\usepackage{cite}
\usepackage{cases}
\usepackage{amsthm}

\def\cl{\centerline}

\def\vs{\vspace*}

\def\Z{\mathbb{Z}}

\def\C{\mathbb{C}}

\numberwithin{equation}{section}
\newtheorem{theo}{Theorem}[section]
\newtheorem{defi}[theo]{Definition}

\newtheorem{lemm}[theo]{Lemma}
\newtheorem{prop}[theo]{Proposition}

\begin{document}
\begin{center}
{\bf\large Loop Schr\"{o}dinger-Virasoro Lie conformal algebra\ $^*$}
\footnote {$^*$ Supported by NSF grant no. 11371278, 11431010, 11301130, 11501417,
the Fundamental Research Funds for the Central Universities of China, Innovation Program of Shanghai Municipal Education Commission and  Program for Young Excellent Talents in Tongji University.

Corresponding author: J. Han (jzhan@tongji.edu.cn).
}
\end{center}

\cl{Haibo Chen$^1$,  Jianzhi
Han$^{1,\dag},$  Yucai Su$^1$, Ying Xu$^2$
}

\cl{\small $^{1}$Department of Mathematics, Tongji University, Shanghai 200092, China}
\cl{\small $^{2}$Department of Mathematics, Hefei University of Technology, Hefei 230009, Anhui, China}
\vs{8pt}

{\small\footnotesize
\parskip .005 truein
\baselineskip 3pt \lineskip 3pt
\noindent{{\bf Abstract:} In this paper, we  introduce two kinds of Lie conformal algebras associated with the loop Schr\"{o}dinger-Virasoro Lie algebra and
the extended loop  Schr\"{o}dinger-Virasoro Lie algebra, respectively. The conformal derivations,  the second cohomology groups of these two conformal algebras are completely determined.  And nontrivial free  conformal modules of rank one and $\Z$-graded free intermediate series modules over these two conformal algebras are also classified in the present paper.
  \vs{5pt}

\parskip .001 truein\baselineskip 6pt \lineskip 6pt

}}
\parskip .001 truein\baselineskip 6pt \lineskip 6pt

\section{Introduction}
Lie conformal algebras encode an axiomatic description of the operator product expansion
of chiral fields in conformal field theory, whose notion was   introduced in \cite{K1,K3}.
They turned out to be an adequate tool for the realization of the program of the study of Lie
(super)algebras and associative algebras (and their representations), satisfying the sole locality property.
 Lie conformal algebras  play important roles in quantum field theory, vertex algebras and integrable systems.
 In addition, Lie conformal algebras have
close connections to Hamiltonian formalism in the theory of nonlinear evolution.
The structure theory and representation theory of Lie conformal algebras were intensively studied
(see, e.g., \cite{DK,BKV,FSW,WCY,CK,S,SY1,{SK},SY2,GXY}).
  In this paper, we investigate derivation algebras, second cohomology groups and some free modules over  conformal algebras associated with  loop algebras of  the Schr\"{o}dinger-Virasoro Lie algebra and its extended algebra.

The Schr\"{o}dinger-Virasoro Lie algebra was introduced in the context of nonequilibrium statistical
physics during the process of investigating the free Schr\"{o}dinger equations in \cite{M1}.
There are two sectors
of this type Lie algebras, i.e., the original one and the twisted one, both of which are closely
related to the Schr\"{o}dinger algebra and the Virasoro algebra, which play important roles in many
areas of mathematics and physics and have been investigated in a series
of papers (see, e.g., \cite{HLS,CFHS,LS1,ZT}). The (extended) loop Schr\"{o}dinger-Virasoro algebra is the Lie algebra of the tensor product of the  (extended) Schr\"{o}dinger-Virasoro algebra and the Laurent polynomial
algebra.
What we  consider in this paper are the loop Schr\"{o}dinger-Virasoro Lie conformal algebra and the extended loop Schr\"{o}dinger-Virasoro Lie conformal algebra. The loop Schr\"{o}dinger-Virasoro algebra $\mathfrak{sv}$ is a Lie algebra with basis $\{L_{m,i},M_{m,i},Y_{p,i}\mid m,i\in\Z,p\in\frac{1}{2}+\Z\}$ subject to the following nontrivial Lie brackets:
 \begin{equation}\label{de1.1}
\aligned
&[L_{m,i},L_{n,j}]= (m-n)L_{m+n,i+j},\ [L_{m,i},M_{n,j}]= -n M_{m+n,i+j},\\
&[L_{m,i},Y_{p,j}]= (\frac{m}{2}-p)
Y_{m+p,i+j},\ [Y_{p,i},Y_{q,j}]=
(p-q)M_{p+q,i+j}
\endaligned
\end{equation}
for any $m,n,i,j\in \Z$ and $p,q\in\frac{1}{2}+\Z$.
The  extended loop Schr\"{o}dinger-Virasoro $\tilde{\mathfrak{sv}}$ is a Lie algebra with basis
$\{L_{m,i},M_{m,i},  Y_{p,i}, N_{m,i}\mid m, i\in\Z, p\in\frac{1}{2}+\Z\}$ satisfying the nontrivial
Lie brackets defined in \eqref{de1.1} and in addition:
\begin{equation}\label{de1.2}
[L_{m,i},N_{n,j}]= -nN_{m+n,i+j},\ [N_{m,i},Y_{p,j}]=
Y_{m+p,i+j},\
[N_{m,i},M_{n,j}]=2M_{m+n,i+j}\
\end{equation}
for any $m,n,i,j\in \Z$ and $p\in\frac{1}{2}+\Z$.

 In this paper,
we   associate  the loop Schr\"{o}dinger-Virasoro Lie algebra and the extended loop Schr\"{o}dinger-Virasoro Lie algebra with two Lie conformal algebras, called the loop Schr\"{o}dinger-Virasoro Lie conformal algebra and the extended loop Schr\"{o}dinger-Virasoro Lie conformal algebra, respectively.
The  {\it loop Schr\"{o}dinger-Virasoro Lie conformal algebra} $\mathfrak{csv}$ has a  $\C[\partial]$-basis $\{L_{i},M_{i},Y_{i}\mid i\in\Z \}$
satisfying the following  nonvanishing $\lambda$-brackets:
\begin{eqnarray}
&&\label{pro3.31}[L_i\, {}_\lambda \, L_j]=(\partial+2\lambda) L_{i+j},\
[L_i\, {}_\lambda \, M_j]=(\partial+\lambda) M_{i+j},\
~~~[M_i\, {}_\lambda \, L_j]=\lambda M_{i+j}, \\
&&\label{pro3.34}
[L_i\, {}_\lambda \, Y_j]=(\partial+\frac{3}{2}\lambda) Y_{i+j},\
[Y_i\, {}_\lambda \, L_j]=(\frac{1}{2}\partial+\frac{3}{2}\lambda) Y_{i+j},\
[Y_i\, {}_\lambda \, Y_j]=(\partial+2\lambda)M_{i+j}
\end{eqnarray}
for any $i,j\in \Z.$
And the  {\it extended loop Schr\"{o}dinger-Virasoro Lie conformal algebra}  $\mathfrak{\tilde{csv}}$ has a  $\C[\partial]$-basis $\{L_{i},M_{i},Y_{i},N_{i}\mid i\in\Z \}$
subject to the nonvanishing $\lambda$-brackets defined in \eqref{pro3.31} and \eqref{pro3.34} together with the following nonvanishing
relations:
\begin{eqnarray}
\label{def1.11}&&[L_i\, {}_\lambda \, N_j]=(\partial+\lambda) N_{i+j},\ \ \
[N_i\, {}_\lambda \, L_j]=\lambda N_{i+j},\
 [N_i\, {}_\lambda \, M_j]=2 M_{i+j}, \\
&&\label{def1.14}
[M_i\, {}_\lambda \, N_j]=-2 M_{i+j},\
~~~~~~[N_i\, {}_\lambda \, Y_j]=Y_{i+j},\
~~~[Y_i\, {}_\lambda \, N_j]=-Y_{i+j}
\end{eqnarray}
for any $i,j\in \Z$.

We remark that the loop Schr\"{o}dinger-Virasoro Lie conformal algebra and the extended loop Schr\"{o}dinger-Virasoro  Lie conformal algebra both contain
\begin{eqnarray*}
\label{cw1.19} \mathcal{CV}=\bigoplus_{i\in\Z}\C[\partial]L_i
\quad {\rm and}\quad  \mathcal{CHV}=\bigoplus_{i\in\Z}(\C[\partial]L_i\oplus\C[\partial]M_i)
\end{eqnarray*}
as their  subalgebras which are isomorphic to the loop Virasoro
Lie conformal algebra  and loop Heisenberg-Virasoro Lie conformal algebra, respectively.
So some results about
these two Lie conformal algebras in \cite{WCY} and\cite{FSW} can be applied in the present paper.

 The rest of this paper
is organized as follows. In Section $2$, we will define $\mathfrak{csv}$ and $\mathfrak{\tilde{csv}}$ after recalling some basic definitions of Lie conformal algebras. In Sections 3 and 4, conformal derivations and second cohomology groups of $\mathfrak{csv}$
and $\mathfrak{\tilde{csv}}$
are investigated. Furthermore, nontrivial free conformal modules of
rank one and $\Z$-graded free intermediate series modules over $\mathfrak{csv}$ and $\mathfrak{\tilde{csv}}$
are classified in Sections 5 and 6, respectively.

Throughout this paper, all vector spaces, linear maps, and tensor products are considered to be over
the field $\C$ of complex numbers. We denote by $\C,$ $\C^*$ and $\Z$ the sets of complex numbers, nonzero complex numbers and integers, respectively. The main results of the present paper are summarized in Theorems \ref{theo4.3}, \ref{theo5.1}, \ref{theo6.1} and \ref{theo7.6}.

\section{Preliminaries}
In this section, we recall some basic definitions and results related to Lie conformal algebras in \cite{DK,K1,K3} for later use.
A {\em formal distribution} with coefficients in a
vector space $U$ is a power series of the following form:
\begin{equation*}
a(z)=\mbox{$\sum\limits_{n\in\Z}$}a_{(n)} z^{-n-1},
\end{equation*}
where $a_{(n)}\in U$ and $z$ is an indeterminate. The vector space of these series is denoted by $U[[z,z^{-1}]]$.
A formal distribution $a(z,w)$  in two variables is similarly defined as a series of the form  $
\mbox{$\sum\limits_{m,n\in\Z}$}a_{(m,n)} z^{-m-1}w^{-n-1},$ and the space of these series is denoted by $U[[z, z^{-1},w,w^{-1}]]$. A formal distribution $a(z,w)\in U[[z,z^{-1},w,w^{-1}]]$ is called {\it local} if there exists some positive integer $N$ such that  $$(z-w)^N a(z,w)=0.$$
Let $\mathfrak{g}$ be a Lie algebra.
Two formal distributions $a(z), b(z)\in \mathfrak{g}[[z,z^{-1}]]$ are
called {\it pairwise local} if $[a(z),b(w)]$ is local in $\mathfrak{g}[[z,z^{-1},w,w^{-1}]]$.
\begin{defi}\label{2.1}{\rm(\!\!\cite{DK})}
A formal distribution Lie algebra is a pair $(\mathfrak{g}, F)$, where $\mathfrak{g}$ is a Lie algebra and $F$ is a family of pairwise local formal distributions whose coefficients $F$ span $\mathfrak{g}$.
\end{defi}
The {\it  delta distribution} is the $\C$-valued formal distribution
\begin{equation*}
\delta(z,w)=\mbox{$\sum\limits_{n\in\Z}$}z^n w^{-n-1}.
\end{equation*}
The following proposition describes an equivalent condition for a formal
distribution to be local.
\begin{prop}\label{pro2.2}{\rm(\!\!\cite{K3})}
A formal distribution $a(z,w)\in U[[z,z^{-1},w,w^{-1}]]\!$ is local if and only if
$a(z,w)$ can be written as
\begin{equation*}
a(z,w)=\mbox{$\sum\limits_{j=0}^\infty$}c^j(w)\frac{\partial^j_w\delta(z,w)}{j!}
\mbox{ \ (finite sum)}
\end{equation*}
for some $c^j(w)\in U[[w,w^{-1}]].$
\end{prop}
\begin{defi}{\rm(i)\ (\!\!\cite{K3})}\label{D1}
A {\it Lie conformal algebra} is a $\C[\partial]$-module $R$ endowed with a  $\lambda$-bracket $[a{}\, _\lambda \, b]$
which defines a
linear map $R\otimes R\rightarrow R[\lambda]$, where $\lambda$ is an indeterminate and $R[\lambda]=\C[\lambda]\otimes R$, subject to the following axioms:
\begin{equation}\label{conformal}
\aligned
&[\partial a\,{}_\lambda \,b]=-\lambda[a\,{}_\lambda\, b],\ \ \ \
[a\,{}_\lambda \,\partial b]=(\partial+\lambda)[a\,{}_\lambda\, b]\quad    {\rm(conformal\ sesquilinearity)},\\
&[a\, {}_\lambda\, b]=-[b\,{}_{-\lambda-\partial}\,a]\quad {\rm (skew\text{-}symmetry)},\\
&[a\,{}_\lambda\,[b\,{}_\mu\, c]]=[[a\,{}_\lambda\, b]\,{}_{\lambda+\mu}\, c]+[b\,{}_\mu\,[a\,{}_\lambda \,c]]\quad{\rm(Jacobi\ identity)},
\endaligned
\end{equation}
for all $a,b,c\in R.$

{\rm (ii)} A Lie conformal algebra $R$ is called {\it $\Z$-graded} if $R=\oplus_{i\in\Z}R_i$, where each $R_i$ is a $\C[\partial]$-submodule
such that $[R_i\,{}_\lambda\, R_j]\subset R_{i+j}[\lambda]$ for any $i,j\in \Z$.
\end{defi}

Note that $\mathfrak{\tilde{csv}}=\oplus_{i\in\Z}\mathfrak{\tilde{csv}}_i$ is $\Z$-graded with $$\mathfrak{\tilde{csv}}_i=\C[\partial]L_i\oplus\C[\partial]M_i\oplus\C[\partial]Y_i\oplus\C[\partial]N_i$$ and $\mathfrak{{csv}}$ is its  $\Z$-graded subalgebra.
\begin{defi}\label{defi-module}{\rm (\!\! \cite{CK})} A  conformal module $M$ over a Lie conformal algebra $R$ is a  $\C[\partial]$-module endowed with a $\lambda$-action $R\otimes M\rightarrow M[\lambda]$ such that
\begin{equation*}
(\partial a)\,{}_\lambda\, v=-\lambda a\,{}_\lambda\, v,\ a{}\,{}_\lambda\, (\partial v)=(\partial+\lambda)a\,{}_\lambda\, v,\
a\,{}_\lambda\, (b{}\,_\mu\, v)-b\,{}_\mu\,(a\,{}_\lambda\, v)=[a\,{}_\lambda\, b]\,{}_{\lambda+\mu}\, v,\end{equation*}
for all $a,b\in R$,  $v\in M$.
\end{defi}

Suppose that $(\mathfrak{g}, F)$ is a formal distribution Lie algebra.
Define
\begin{equation*}\label{abw2.3}
[a(w)\,{}_\lambda \, b(w)]=F^\lambda_{z,w}[a(z),b(w)]
\end{equation*}
for any $a(w),b(w)\in F$.  Then \begin{eqnarray}\label{abw2.-3}
[a(w)\,{}_\lambda \, b(w)]=F^\lambda_{z,w}[a(z),b(w)]={\rm Res}_z e^{\lambda(z-w)}\mbox{$\sum\limits_{j=0}^\infty$}c^j(w)\frac{\partial^j_w\delta(z,w)}{j!}=\sum_{j=0}^\infty \frac{\lambda^j}{j!}c^j(w),
\end{eqnarray} where we have assumed $[a(z),b(w)]=\mbox{$\sum\limits_{j=0}^\infty$}c^j(w)\frac{\partial^j_w\delta(z,w)}{j!}$ for some $c^j(w)\in\mathfrak g[[w,w^{-1}]]$ by Proposition \ref{pro2.2} and used the fact $\frac{1}{(j+1)!}(z-w)\partial_w^{j+1}\delta(z,w)=\frac{1}{j!}\partial_w^j\delta(z,w)$ for any nonnegative integer $j$.   One can easily check that these \mbox{$\lambda$}-brackets satisfy relations  (\ref{conformal}).
Given a formal distribution Lie algebra $(\mathfrak{g}, F)$, we may always include $F$ in the minimal family $F^c$ of pairwise local distributions which is closed under
the derivative $\partial$ \big(the action of $\partial$ on an element $a(w)\in F^c$ is given by $\partial_w a(w)=(\partial a)(w)$\big) and the $\lambda$-brackets.
An important fact worthwhile  to  point out is that $(\mathfrak g, F^c)$  actually forms a Lie conformal algebra of $\mathfrak{g}.$

Set \begin{eqnarray*}&L_i(z)=\sum_{n\in\Z}L_{n,i}z^{-n-2},\quad\ \ \ M_i(z)=\sum_{n\in\Z}M_{n,i}z^{-n-1},\\
 &Y_i(z)=\sum_{p\in\frac{1}{2}+\Z}Y_{p,i}z^{-p-\frac{3}{2}},\quad N_i(z)=\sum_{n\in\Z}N_{n,i}z^{-n-1}\end{eqnarray*} for any $i\in\Z.$
Then a straightforward computation by using \eqref{de1.1} and \eqref{de1.2} shows
\begin{equation}\label{csv2.2}
\aligned
&[L_i(z),L_j(w)]=(\partial_w L_{i+j}(w))\delta(z,w)+2L_{i+j}(w)\partial_w\delta(z,w),\\
&
[L_i(z),M_j(w)]=(\partial_w M_{i+j}(w))\delta(z,w)+M_{i+j}(w)\partial_w\delta(z,w),\\
&
[L_i(z),Y_j(w)]=(\partial_w Y_{i+j}(w))\delta(z,w)+\frac{3}{2}Y_{i+j}(w)\partial_w\delta(z,w),\\&
[Y_i(z),Y_j(w)]=(\partial_w M_{i+j}(w))\delta(z,w)+2M_{i+j}(w)\partial_w\delta(z,w),\\&
[L_i(z),N_j(w)]=(\partial_wN_{i+j}(w))\delta(z,w)+N_{i+j}(w)\partial_w\delta(z,w),\\&
[N_i(z),M_j(w)]=2M_{i+j}(w)\delta(z,w),\\&
[N_i(z),Y_j(w)]=Y_{i+j}(w)\delta(z,w),\quad\forall\ i,j\in\Z.
\endaligned
\end{equation}
It follows from  \eqref{abw2.-3} and \eqref{csv2.2} that
\begin{eqnarray*}
&&[L_i(w)\, {}_\lambda \, L_j(w)]=(\partial+2\lambda) L_{i+j}(w),\quad
[L_i(w)\, {}_\lambda \, M_j(w)]=(\partial+\lambda) M_{i+j}(w),\\
&&[L_i(w)\, {}_\lambda \, Y_j(w)]=(\partial+\frac{3}{2}\lambda) Y_{i+j}(w),\quad
[Y_i(w)\, {}_\lambda \, Y_j(w)]=(\partial+ 2\lambda)M_{i+j}(w),\\
&&[L_i(w)\, {}_\lambda \, N_j(w)]=(\partial+\lambda)N_{i+j}(w),\quad
[N_i(w)\, {}_\lambda \, M_j(w)]=2 M_{i+j}(w),\\
&&[N_i(w)\, {}_\lambda \, Y_j(w)]=Y_{i+j}(w),\quad \forall\ i,j\in\Z,
\end{eqnarray*}which are consistent with relations \eqref{pro3.31}-\eqref{def1.14}.

\section{Conformal derivations of $\mathfrak{csv}$ and $\mathfrak{\tilde{csv}}$}
In this section, we first recall some definitions related to conformal derivations.
Let $R$ be a Lie conformal algebra.
Then a linear map $d_\lambda: R\rightarrow R[\lambda]$ is called a {\em conformal derivation} of $R$ if
\begin{equation*}
\aligned
d_\lambda(\partial a)=(\partial+\lambda)d_\lambda a,\ \
d_\lambda([a\,{}_\mu \,b])=[(d_\lambda a)\,{}_{\lambda+\mu} \,b]+[a\,{}_\mu \,(d_\lambda b)]\quad{\rm for\ any}\ a, b\in R.
\endaligned
\end{equation*}
For simplicity, we denote $d_\lambda$ by $d.$
For any $x\in R$, we can easily get the map ${\rm ad}_x$ defined by $({\rm ad}_x)_\lambda y= [x\, {}_\lambda\, y]$ for $y\in R$ is a conformal derivation of $R$.
All conformal derivations of this kind are called {\it inner}.
Denote by ${\rm Der\,}(R)$
and ${\rm Inn\,}(R)$ the vector spaces of all conformal derivations and inner conformal derivations of $R$, respectively.

Now we aim to compute conformal derivations of the loop Schr\"{o}dinger-Virasoro Lie
conformal algebra $\mathfrak{csv}$ and the extended loop Schr\"{o}dinger-Virasoro Lie
conformal algebra $\mathfrak{\tilde{csv}}.$

Denote $$\C^\infty=\{\vec{a}=(a_c)_{c\in\Z}\mid a_c\in\C\mbox{ and }a_c=0\ \mbox{for\ all\ but\ finitely\ many}\ c\mbox{'}s\}.$$
For each $\vec{a}\in \C^\infty$, we define $$D_{\vec{a}}\, {}_\lambda\, (L_i)=\sum_{c\in\Z} a_c M_{i+c},\ D_{\vec{a}}\, {}_\lambda\,(M_i)=0\ {\rm and}\ D_{\vec{a}}\, {}_\lambda\,(Y_i)=0\quad {\rm for\ all} \ i\in\Z.$$
Clearly, $D_{\vec{a}}\in {\rm Der\,}(\mathfrak{csv})$, and we denote  the set of  all  such conformal derivations by $\C^\infty$.
 It  is easy  to verify that $D_{\vec{a}}\in {\rm Inn\,}(\mathfrak{csv})$ implies $D_{\vec{a}}=0$, which is equivalent to saying ${\rm Inn\,}(\mathfrak{csv})\cap \C^\infty=0$.

 Let $R$ be a $\Z$-graded Lie conformal algebra.
  For any $D\in{\rm Der\,}({R})$ and $c\in\Z$,  define
 $$D^c(x_i)=\pi_{i+c} D(x_i),\quad\forall\ x_i\in R_i,$$ where  $\pi_{k}$ is the natural projection from $\C[\lambda]\otimes {R}$ onto $\C[\lambda]\otimes {R}_k$.
Then $D^c$ is a conformal derivation and $D=\sum_{c\in\Z} D^c$ in the sense that for any $x\in R$ only finitely many $D^c_\lambda(x)\neq0$.
Let ${({\rm Der\,}({R}))}^c$ be the space of conformal derivations of degree $c$, i.e.,
 $${\big({\rm Der\,}({R})\big)}^c=\Big\{D\in {\rm Der\,}({R})\mid D_\lambda({R}_i)\subset {R}_{i+c}[\lambda],\ \ \ \forall\, i\in\Z\Big\}.$$

\begin{lemm}\label{lc1}
 For any $D^c\in {\big({\rm Der\,}(\mathfrak{csv})\big)}^c$, there exists an $a_c\in \C$  such that $D^c-D_{\vec a}\in {\rm Inn\,}(\mathfrak{csv})$, where $\vec a\in \C^\infty$ with the $c$-th entry $a_c$ and zero elsewhere.
\end{lemm}
\begin{proof}   Assume \begin{equation}\label{expres}D^c_\lambda(L_i)=f_{1,i}(\partial,\lambda)L_{i+c}+f_{2,i}(\partial,\lambda)M_{i+c}+f_{3,i}(\partial,\lambda)Y_{i+c}\end{equation} for any $i\in\Z$ and some $f_{j,i}(\partial,\lambda)\in\C[\partial,\lambda]$.
It follows from applying $D^c_\lambda$ to $[L_0\, {}_\mu \, L_i]$  $=(\partial+2\mu)L_i$ that
\begin{eqnarray*}
&&(\partial+\lambda+2\mu)(f_{1,i}(\partial,\lambda)L_{i+c}+f_{2,i}(\partial,\lambda)M_{i+c}+f_{3,i}(\partial,\lambda)Y_{i+c})\\
&=&\big((\partial+2\lambda+2\mu)f_{1,0}(-\lambda-\mu,\lambda)+(\partial+2\mu)f_{1,i}(\partial+\mu,\lambda)\big)L_{i+c}+\\
&&
 \big((\lambda+\mu)f_{2,0}(-\lambda-\mu,\lambda)+(\partial+\mu)f_{2,i}(\partial+\mu,\lambda)\big)M_{i+c}+\\
&&
\big((\frac{1}{2}\partial+\frac{3}{2}\lambda+\frac{3}{2}\mu)f_{3,0}(-\lambda-\mu,\lambda)+(\partial+\frac{3}{2}\mu)f_{3,i}(\partial+\mu,\lambda)\big)Y_{i+c},\end{eqnarray*}
from which by  respectively comparing the coefficients of $L_{i+c},$ $M_{i+c}$ and $Y_{i+c}$ one has
\begin{eqnarray*}\label{c1}
&&(\partial+\lambda+2\mu)f_{1,i}(\partial,\lambda)=(\partial+2\lambda+2\mu)f_{1,0}(-\lambda-\mu,\lambda)+(\partial+2\mu)f_{1,i}(\partial+\mu,\lambda),\\
&&
\label{c2}
(\partial+\lambda+2\mu)f_{2,i}(\partial,\lambda)=(\lambda+\mu)f_{2,0}(-\lambda-\mu,\lambda)+(\partial+\mu)f_{2,i}(\partial+\mu,\lambda),\\
&&
\label{c3}
(\partial+\lambda+2\mu)f_{3,i}(\partial,\lambda)
=\frac{1}{2}(\partial+3\lambda+3\mu)f_{3,0}(-\lambda-\mu,\lambda)+
(\partial+\frac{3}{2}\mu)f_{3,i}(\partial+\mu,\lambda).
\end{eqnarray*}
In particular, setting $\mu=0$  in above three relations, we correspondingly  have:
\begin{eqnarray}
\label{4.2}&& \lambda f_{1,i}(\partial,\lambda)=(\partial+2\lambda)f_{1,0}(-\lambda,\lambda),\\
\label{4.3}&&\lambda f_{2,i}(\partial,\lambda)=\lambda f_{2,0}(-\lambda,\lambda),\\
\label{4.4}&&
\lambda f_{3,i}(\partial,\lambda)=(\frac{1}{2}\partial+\frac{3}{2}\lambda) f_{3,0}(-\lambda,\lambda).
\end{eqnarray}
Note that  if we set $F_1(\lambda)=\frac{f_{1,0}(\lambda,-\lambda)}{\lambda}$ and $F_3(\lambda)=\frac{f_{3,0}(\lambda,-\lambda)}{\lambda},$ then  from  \eqref{4.2} and \eqref{4.4} that both  $F_1(\lambda)$ and $F_3(\lambda)$ are elements of $\C[\lambda]$. Write $f_{2,0}(-\lambda,\lambda)=\lambda H(\lambda)+a_c$ for some $H(\lambda)\in\C[\lambda]$ and $a_c\in\C$.
 Then it follows from   replacing $D^c$ by $$D^c+{\rm ad}_{{F_1(\partial)}L_c}+{\rm ad}_{{F_3(\partial)}Y_c}-{\rm ad}_{{H(-\partial)}M_c}$$ and relations \eqref{4.2}-\eqref{4.4} we see that \eqref{expres} simply turns out to be
 $$D^c_\lambda(L_i)=a_cM_{i+c}\quad {\rm for\ any}\ i\in \Z.$$

 Assume $$D^c_\lambda(M_i)=g_{1,i}(\partial,\lambda)L_{i+c}+g_{2,i}(\partial,\lambda)M_{i+c}+g_{3,i}(\partial,\lambda)Y_{i+c}$$ for some $g_{j,i}(\partial, \lambda)\in\C [\partial,\lambda]$.  Applying $D^c_\lambda$ to $[L_0\, {}_\mu \, M_i]=(\partial+\mu)M_i$ gives rise to
\begin{eqnarray*}
&&(\partial+\lambda+\mu)(g_{1,i}(\partial,\lambda)L_{i+c}+g_{2,i}(\partial,\lambda)M_{i+c}+g_{3,i}(\partial,\lambda)Y_{i+c})\\
&=&(\partial+2\mu)g_{1,i}(\partial+\mu,\lambda)L_{i+c}+(\partial+\mu)g_{2,i}(\partial+\mu,\lambda)M_{i+c}+(\partial+\frac{3}{2}\mu)g_{3,i}(\partial+\mu,\lambda)
Y_{i+c},
\end{eqnarray*}
which is equivalent to the following three relations:
\begin{eqnarray}\label{4.8}
(\partial+\lambda+\mu)g_{1,i}(\partial,\lambda)&\!\!\!=\!\!\!&(\partial+2\mu)g_{1,i}(\partial+\mu,\lambda),\\
\label{4.9}
(\partial+\lambda+\mu)g_{2,i}(\partial,\lambda)&\!\!\!=\!\!\!&(\partial+\mu)g_{2,i}(\partial+\mu,\lambda),\\
\label{4.10}
(\partial+\lambda+\mu)g_{3,i}(\partial,\lambda)&\!\!\!=\!\!\!&(\partial+\frac{3}{2}\mu)g_{3,i}(\partial+\mu,\lambda), \quad \forall\ i\in\Z.
\end{eqnarray}
Comparing the highest degrees of monomials with respect to $\lambda$ in \eqref{4.8}-\eqref{4.10}
one can immediately get
$$g_{1,i}(\partial,\lambda)=g_{2,i}(\partial,\lambda)=g_{3,i}(\partial,\lambda)=0,$$ namely,  $D^c_\lambda(M_i)=0$ for any $i\in\Z$. Similarly, one can show that $D^c_\lambda(Y_i)=0$ for any $i\in\Z$. This completes the proof.
\end{proof}

Now we are ready to state  the main result of this section.
\begin{theo}\label{theo4.3} We have the following decompositions:

{\rm (1)} ${\rm Der\,}  (\mathfrak{csv})={\rm Inn\,}(\mathfrak{csv})\oplus \C^\infty;$

 {\rm (2)} ${\rm Der\,}(\mathfrak{\tilde{csv}})={\rm Inn\,}(\mathfrak{\tilde{csv}}).$
\end{theo}

\begin{proof} (1) For any $D=\sum_{c\in\Z}D^c\in{\rm Der\,}(\mathfrak{csv})$ with $D^c\in\big({\rm Der\,}(\mathfrak{csv})\big)^c$, to prove $D\in {\rm Inn\,}(\mathfrak{csv})\oplus \C^\infty$,  it is enough to show that the sum $\sum_{c\in\Z}D^c$ is finite, since  by Lemma \ref{lc1} $${\big({\rm Der\,}(\mathfrak{csv})\big)}^c\subseteq {\rm Inn\,}(\mathfrak{csv})\oplus \C^\infty,\quad \forall\ c\in\Z.$$ Suppose not, i.e., $J=\{c\in\Z\mid D^c\neq0\}$ is an infinite set. Due to Lemma \ref{lc1}, we can write  each $$D^c=\mbox{ad}_{\mu_c(\partial)L_c+\nu_c(\partial)M_c+\omega_c(\partial)Y_c}+D_{\vec a}$$ for some $\mu_c(\partial),\nu_c(\partial),\omega_c(\partial) \in\C[\partial]$ and $\vec a\in\C^\infty$ with the only possible  nonzero $c$-th entry $a_c$.  Note that $$D^c{}_\lambda(L_0)=(\partial+2\lambda)\mu_c(-\lambda)L_c+\lambda \nu_c(-\lambda)M_c+(\frac{1}{2}\partial+\frac{3}{2}\lambda)\omega_c(-\lambda)Y_c+a_c M_c\neq 0\quad {\rm for\ any}\ c\in J.$$ Hence,
\begin{eqnarray*}
D{}_\lambda(L_0)=\sum_{c\in\Z}(\partial+2\lambda)\mu_c(-\lambda)L_c+\lambda \nu_c(-\lambda)M_c+(\frac{1}{2}\partial+\frac{3}{2}\lambda)\omega_c(-\lambda)Y_c+a_c M_c
\end{eqnarray*}   is an infinite sum, which is not an element in $\mathfrak{csv}$, contradicting
the fact that $D$ is a  linear map from $\mathfrak{csv}$ to $\mathfrak{csv}$.

(2) For this, it suffices to show $\big({\rm Der\,}(\mathfrak{\tilde{csv}})\big)^c\subseteq {\rm Inn}(\mathfrak{\tilde{csv}})$ for any $c\in\Z$.  Let $D^c\in\big({\rm Der\,}(\mathfrak{\tilde{csv}})\big)^c.$
It follows from the proof of Lemma \ref{lc1} that we may assume
\begin{eqnarray}\label{eq4.14}
&D^c_\lambda(L_i)=a_cM_{i+c}+r_i(\partial,\lambda)N_{i+c}   \quad {\rm for \ some }\ a_c\in\C,\,r_i(\partial,\lambda)\in\C[\partial,\lambda],\\
&\label{eq4.1555}D^c_\lambda(M_i)=m_i(\partial,\lambda)N_{i+c}\quad {\rm for \ some }\ m_i(\partial,\lambda)\in\C[\partial,\lambda].
\end{eqnarray}

Applying $D^c_\lambda$ to $[L_0\, {}_\mu \, L_i]=(\partial+2\mu) L_{i}$, we have
\begin{eqnarray*}
&&(\partial+\lambda+2\mu)\big(a_cM_{i+c}+r_{i}(\partial,\lambda)N_{i+c}\big)\\
&\!\!\!=\!\!\!&(\partial+\lambda+2\mu)a_cM_{i+c}+
\big((\lambda+\mu)r_0(-\lambda-\mu,\lambda)+(\partial+\mu)r_i(\partial+\mu,\lambda)\big)N_{i+c},
\end{eqnarray*}which forces
\begin{equation}\label{partia3.11}
(\partial+\lambda+2\mu)r_{i}(\partial,\lambda)=(\lambda+\mu)r_0(-\lambda-\mu,\lambda)+(\partial+\mu)r_i(\partial+\mu,\lambda),\quad \forall\ i\in\Z.
\end{equation}
Taking $\mu=0$ in \eqref{partia3.11}, we have $r_i(\partial,\lambda)=r_0(-\lambda,\lambda)$ for all $i\in\Z.$
Hence, \eqref{eq4.14} can be rephrased as
\begin{equation*}
D^c_\lambda(L_i)=a_cM_{i+c}+r_0(-\lambda,\lambda)N_{i+c} \quad {\rm for \ all }\ i\in\Z.
\end{equation*}It follows from this and applying $D^c_\lambda$ to $[L_0\, {}_\mu \, M_i]=(\partial+\mu) M_{i}$ that $$ 2r_0(-\lambda,\lambda)+(\partial+\mu)m_i(\partial+\mu,\lambda)=(\partial+\lambda+\mu)m_i(\partial,\lambda),$$
which implies \begin{eqnarray}\label{rmmi3.10}
\frac{r_0(-\lambda,\lambda)}{\lambda}=\frac{1}{2} m_i(\partial,\lambda)\in\C[\lambda] \quad {\rm for \ all } \ i\in\Z.
\end{eqnarray}
 Then it follows from   replacing $D^c$ by $$D^c+{\rm ad}_{\frac{r_0(\partial,-\partial)}{\partial}N_c}$$ and relation \eqref{rmmi3.10} we see that \eqref{eq4.14} and \eqref{eq4.1555} simply turn out to be
 $$D^c_\lambda(L_i)=a_cM_{i+c}\quad {\rm and}\quad D^c_\lambda(M_i)=0,\quad  \forall \ i\in\Z.$$

We are going to show $D^c=0$, hence completing the proof. Now assume
  \begin{equation}\label{dcni3.12}
  D^c_\lambda(N_i)=b_{1,i}(\partial,\lambda)L_{i+c}+b_{2,i}(\partial,\lambda)M_{i+c}+b_{3,i}(\partial,\lambda)Y_{i+c}+b_{4,i}(\partial,\lambda)N_{i+c}
  \end{equation}
  for any $i\in\Z$ and some $b_{j,i}(\partial,\lambda)\in\C[\partial,\lambda]$.
It follows from applying $D^c_\lambda$ to $[L_0\, {}_\mu \, N_i]=(\partial+\mu) N_{i}$  we see that \begin{equation*}
(\partial+\lambda+\mu)D_\lambda^c(N_{i})=[L_0\, {}_\mu\, D_\lambda^c (N_i)]-2a_cM_{i+c},\quad\forall \ i\in\Z,
\end{equation*}
from which we deduce that \begin{equation}\label{eq-bj}
(\partial+\lambda+\mu)b_{2,i}(\partial,\lambda)=(\partial+\mu)b_{2,i}(\partial+\mu,\lambda)-2a_c,\quad\forall\ i\in\Z,\end{equation} where $b_{2,i}(\partial,\lambda)\in\C[\partial,\lambda]$ is  assumed to be
$$D_\lambda^c(N_i)\equiv b_{2,i}(\partial,\lambda)M_{i+c} \ \big({\rm mod} \bigoplus_{i\in\Z,\ X\in\{L, Y, N\}}\C[\partial,\lambda]X_i \big).$$
Comparing  the highest degrees of both sides of \eqref{eq-bj} with respect to $\lambda$ gives $b_{2,i}(\partial,\lambda)=0$ for any $i\in\Z$ and then $a_c=0$. Similarly, one can show that $$D_\lambda^c(N_i)\equiv 0 \ \big({\rm mod} \bigoplus_{i\in\Z,\ X\in\{L, M, Y, N\}\setminus\{S\}}\C[\partial,\lambda]X_i \big)\quad{\rm for\ any}\ S\in\{L, Y, N\}.$$ To sum up,  we have shown $$D^c_\lambda(L_i)=D^c_\lambda(M_i)=D^c_\lambda(N_i)=0,\quad \forall\ i\in\Z.$$
Now it is easy to show that $D^c_\lambda(Y_i)=0$ for any $i\in\Z$. Hence, $D^c=0$ for any $c\in\Z$.\end{proof}

\section{Second cohomology groups of $\mathfrak{csv}$ and $\mathfrak{\tilde{csv}}$}
Throughout this section,  we consider the second cohomology groups of   $\mathfrak{csv}$ and  $\mathfrak{\tilde{csv}}$ and {\em by a complex function we always mean a map from $\Z$ to $\C$.}

We shall briefly recall  the definition of the second cohomology group (see \cite{BKV,S,SK}). Let $R$ be a Lie conformal algebra.
A bilinear map $\phi_\lambda: R\times R \rightarrow \C[\lambda]$ is called   a {\it 2-cocycle}
on $R$ if the following conditions are satisfied:
\begin{equation*}
\aligned
&\phi_{\lambda}(a, b)=- \phi_{-\lambda}(b, a)\quad {\rm (skew\text{-}symmetry)},\\
&
\phi_{\lambda}(\partial a, b)=-\lambda \phi_{\lambda}(a, b)=-\phi_{\lambda}(a,\partial b)\quad {\rm (conformal \ sesquilinearity)},\\
&\phi_{\lambda+\mu}([a\,_\lambda \, b], c)=\phi_{\lambda}(a, [b\,_\mu\, c])-\phi_{\mu}(b, [a\,_\lambda\, c])\quad{\rm(Jacobi\ identity)}
\endaligned
\end{equation*}
for all  $a,b,c\in R.$
Denote by $C^{2}(R,\ \C)$ the vector space of
2-cocycles on
 $R$.
For any linear map $f:R\rightarrow \C[\lambda]$, one can define a 2-cocycle $df$ as follows:
\begin{eqnarray}\label{eq4.1}
&&(df)\, {}_\lambda\,(x,y)=-f([x\, {}_\lambda\,y]), \quad\forall\ x,y \in R.
\end{eqnarray}
Such a 2-cocycle is called a {\it 2-coboundary} on $R$. Denote by
$B^{2}(R,\ \C)$ the vector space of 2-coboundaries on
 $R$. The quotient space
 \begin{eqnarray*}
&&H^{2}(R,\ \C)=C^{2}(R,\ \C)/B^{2}(R,\
\C)
\end{eqnarray*}
is called  the {\it second cohomology group} of $R$. It is  well-known fact that the second cohomology group of $R$
is closely related to  central extensions of
$R$. To be more precise,  there is a one-to-one correspondence between the
set of equivalence classes of one-dimensional central extensions of
$R$ by $\C$ and the second cohomology group of $R$.

We are going to determine the second cohomology group of  $\mathfrak{csv}$.
Now let $\phi^\prime$ be a 2-cocycle of  $\mathfrak{csv}$ (resp. $\mathfrak{\tilde{csv}}$). Define a
linear map $f$ from $\mathfrak{csv}$ (resp. $\mathfrak{\tilde{csv}}$)  to $\mathbb{C}$ as follows:
\begin{eqnarray*}
&&f(L_i)=\frac{1}{2}\frac{d}{d\lambda}\phi^\prime _\lambda(L_i,L_0)|_{\lambda=0},\
f(M_i)=\frac{d}{d\lambda}\phi^\prime_\lambda(M_i,L_0)|_{\lambda=0},\\
\label{eq-f22}&&f(Y_i)=\frac{2}{3}\frac{d}{d\lambda}\phi^\prime_\lambda(Y_i,L_0)|_{\lambda=0}\ ({\rm resp.}\ f(N_i)=\frac{d}{d\lambda}\phi^\prime _\lambda(N_i,L_0)|_{\lambda=0}),
 \quad \forall\ i\in \Z.
\end{eqnarray*}
Set $\phi=\phi^\prime+df$, where $df$ is defined in
\eqref{eq4.1}.
For any $i\in\Z$, we have
\begin{eqnarray*}
&&\frac{d}{d\lambda}\phi_\lambda(L_i,L_0)|_{\lambda=0}\nonumber\\ &=&\frac{d}{d\lambda}\phi^\prime_\lambda(L_i,L_0)|_{\lambda=0}+
\frac{d}{d\lambda}(df)_\lambda(L_i,L_0)|_{\lambda=0}\nonumber\\
&\!\!\!=\!\!\!&\frac{d}{d\lambda}\phi^\prime_\lambda(L_i,L_0)|_{\lambda=0}-\frac{1}{2}
\frac{\partial}{\partial\lambda}\big(\frac{\partial}{\partial\mu}\phi^\prime_\mu([L_i\, {}_\lambda\,L_0],L_0)|_{\mu=0}\big)|_{\lambda=0}\\
&\!\!\!=\!\!\!&\frac{d}{d\lambda}\phi^\prime_\lambda(L_i,L_0)|_{\lambda=0}-\frac{1}{2}
\frac{\partial}{\partial\lambda}\big(\frac{\partial}{\partial\mu}((2\lambda-\mu)\phi^\prime_\mu(L_i,L_0))|_{\mu=0}\big)|_{\lambda=0}\nonumber\\
&\!\!\!=\!\!\!&\frac{d}{d\lambda}\phi^\prime_\lambda(L_i,L_0)|_{\lambda=0}-\frac{1}{2}
\frac{\partial}{\partial\mu}\big(\frac{\partial}{\partial\lambda}((2\lambda-\mu)\phi^\prime_\mu(L_i,L_0))|_{\lambda=0}\big)|_{\mu=0}
=0.\nonumber
\end{eqnarray*}
Similarly,
\begin{equation}\label{dmy4.5}
\frac{d}{d\lambda}\phi_\lambda(M_i,L_0)|_{\lambda=0}=\frac{d}{d\lambda}\phi_\lambda(Y_i,L_0)|_{\lambda=0}=0\ ({\rm resp.}\ \frac{d}{d\lambda}\phi_\lambda(N_i,L_0)|_{\lambda=0}=0), \quad\forall\ i\in\Z.
\end{equation}

 Using the Jacobi identity on $(L_{i},L_{j},L_{k})$, we  obtain
 \begin{eqnarray}\label{lammu44.33}
&&(\lambda-\mu)\phi_{\lambda+\mu}(L_{i+j}, L_{k})=(\lambda+2\mu)\phi_{\lambda}(L_{i}, L_{j+k})-(\mu+2\lambda)\phi_{\mu}(L_{j}, L_{i+k})
\end{eqnarray}for any $i,j,k\in\Z$.
Setting $k=0$ in \eqref{lammu44.33} gives rise to
\begin{equation}\label{0par4.4}
\aligned
0&=\frac{\partial}{\partial\mu}\big((\lambda-\mu)\phi_{\lambda+\mu}(L_{i+j}, L_{0})-(\lambda+2\mu)\phi_{\lambda}(L_{i}, L_{j})+(\mu+2\lambda)\phi_{\mu}(L_{j}, L_{i})\big)|_{\mu=-\lambda}\\
&=-\phi_{0}(L_{i+j}, L_{0})+2\lambda\frac{\partial}{\partial\mu}\phi_{\lambda+\mu}(L_{i+j}, L_{0})|_{\mu=-\lambda}+\\&
~~~~\frac{\partial}{\partial\mu}\big(-(\lambda+2\mu)\phi_{\lambda}(L_{i}, L_{j})+(\mu+2\lambda)\phi_{\mu}(L_{j}, L_{i})\big)|_{\mu=-\lambda}\\
&=-\phi_{0}(L_{i+j}, L_{0})-3\phi_{\lambda}(L_{i}, L_{j})+\lambda\frac{d}{d\lambda}\phi_{\lambda}(L_{i}, L_{j}).
\endaligned
\end{equation}
Setting $\lambda=j=0$ in  \eqref{0par4.4}, one has $\phi_{0}(L_{i}, L_{0})=0$ for all $i\in\Z$.
Then  \eqref{0par4.4} implies $3\phi_{\lambda}(L_{i}, L_{j})=\lambda\frac{d}{d\lambda}\phi_{\lambda}(L_{i}, L_{j})$,
from which we solve
\begin{eqnarray}\label{phla4.4}
\phi_{\lambda}(L_{i}, L_{j})=a_{i,j}\lambda^3
\end{eqnarray} for some $a_{i,j}\in\C$ and all $i,j\in\Z$.
Inserting \eqref{phla4.4} into \eqref{lammu44.33} and comparing the coefficients of $\lambda^4$ and $\mu^4$, one has
$$a_{i+j,k}=a_{i,j+k}=a_{j,i+k},\quad \forall\ i,j,k\in\Z.$$
This shows that the values $a_{i,j}$  depend only on the sum $i+j$. Then
this allows us to write
\begin{equation*}
\phi_{\lambda}(L_{i}, L_{j})=a_{i+j}\lambda^3,
\end{equation*}
where $a$ is a complex function given by $a_{i+j}:=a_{i,j}$ for all $i,j\in\Z$.

In order to determine $\phi_{\lambda}(M_{i},L_{j})$  and $\phi_{\lambda}(Y_{i},L_{j}),$  we apply the Jacobi identity to triples $(M_{i},L_{j},L_{0})$ and $(Y_{i},L_{j},L_{0})$, respectively,  yielding
\begin{eqnarray}\label{lambda44.4}
0&\!\!\!=\!\!\!&\frac{\partial}{\partial\mu}\big(\lambda\phi_{\lambda+\mu}(M_{i+j}, L_{0})-(\lambda+2\mu)\phi_{\lambda}(M_{i}, L_{j})+\lambda\phi_{\mu}(L_{j}, M_{i})\big)|_{\mu=-\lambda},\\
\label{lambda44.5}
0&\!\!\!=\!\!\!&\frac{\partial}{\partial\mu}\big((\lambda-\frac{1}{2}\mu)\phi_{\lambda+\mu}(Y_{i+j}, L_{0})
-(\lambda+2\mu)\phi_{\lambda}(Y_{i}, L_{j})+\nonumber\\&&(\frac{1}{2}\mu+\frac{3}{2}\lambda)\phi_{\mu}(L_{j}, Y_{i})\big)|_{\mu=-\lambda}.
\end{eqnarray}
Combining \eqref{lambda44.4} with \eqref{dmy4.5} one has
\begin{eqnarray*}
0&\!\!\!=\!\!\!&\frac{\partial}{\partial\mu}\big(-(\lambda+2\mu)\phi_{\lambda}(M_{i}, L_{j})+\lambda\phi_{\mu}(L_{j}, M_{i})\big)|_{\mu=-\lambda}\\
&\!\!\!=\!\!\!&-2\phi_{\lambda}(M_{i}, L_{j})+\lambda\frac{d}{d\lambda}\phi_{\lambda}(M_{i},L_{j}),
\end{eqnarray*}
from which we can infer that
\begin{eqnarray}\label{myb4.9}
\phi_{\lambda}(M_{i}, L_{j})=b_{i,j}\lambda^2
\end{eqnarray}
for some $b_{i,j}\in\C$ and all $i,j\in\Z$.
Similarly,    \eqref{lambda44.5} together with \eqref{dmy4.5} leads to
\begin{equation}\label{phiyl4.9}
\aligned
0&=-\frac{1}{2}\phi_{0}(Y_{i+j},L_0)+\frac{3}{2}\lambda\frac{\partial}{\partial\mu}\phi_{\lambda+\mu}(Y_{i+j}, L_{0})|_{\mu=-\lambda}+\\&
~~~\frac{\partial}{\partial\mu}\big(-(\lambda+2\mu)\phi_{\lambda}(Y_{i}, L_{j})+(\frac{1}{2}\mu+\frac{3}{2}\lambda)\phi_{\mu}(L_{j}, Y_{i})\big)|_{\mu=-\lambda}\\
&=-\frac{1}{2}\phi_{0}(Y_{i+j},L_0)-\frac{5}{2}\phi_{\lambda}(Y_{i}, L_{j})+\lambda\frac{d}{d\lambda}\phi_{\lambda}(Y_{i}, L_{j}).
\endaligned
\end{equation}
Setting $\lambda=j=0$  in \eqref{phiyl4.9}, one has $\phi_{0}(Y_{i},L_0)=0$ for all $i\in\Z$.
Then \eqref{phiyl4.9} can be rephrased as $\frac{5}{2}\phi_{\lambda}(Y_{i}, L_{j})=\lambda\frac{d}{d\lambda}\phi_{\lambda}(Y_{i}, L_{j})$.
Note that this differential equation has only trivial  solution in $\C[\lambda]$. Hence,  $$\phi_{\lambda}(Y_{i}, L_{j})=0,\quad \forall\ i,j\in\Z.$$

It follows from applying the Jacobi identity to the triple $(Y_i,Y_j,L_k)$  that
\begin{eqnarray}\label{my4.10}
(\lambda-\mu)\phi_{\lambda+\mu}(M_{i+j}, L_{k})=(\frac{1}{2}\lambda+\frac{3}{2}\mu)\phi_{\lambda}(Y_{i}, Y_{j+k})-(\frac{1}{2}\mu+\frac{3}{2}\lambda)\phi_{\mu}(Y_{j},Y_{i+k})
\end{eqnarray}
for any $i,j,k\in\Z$.
In particular, let $k=0$ in  \eqref{my4.10}, by \eqref{dmy4.5},  \eqref{myb4.9} and the skew-symmetry
we have
\begin{eqnarray*}
0&\!\!\!=\!\!\!&\frac{\partial}{\partial\mu}\big(-(\frac{1}{2}\lambda+\frac{3}{2}\mu)\phi_{\lambda}(Y_{i}, Y_{j})+(\frac{1}{2}\mu+\frac{3}{2}\lambda)\phi_{\mu}(Y_{j}, Y_{i})\big)|_{\mu=-\lambda}\\
&\!\!\!=\!\!\!&-2\phi_{\lambda}(Y_{i}, Y_{j})+\lambda\frac{d}{d\lambda}\phi_{\lambda}(Y_{i}, Y_{j}).
\end{eqnarray*}
Then we can conclude that
\begin{eqnarray}\label{my4.11}
\phi_{\lambda}(Y_{i}, Y_{j})=c_{i,j}\lambda^2
\end{eqnarray} for some $c_{i,j}\in\C$ and all $i,j\in\Z$.
Due to the skew-symmetry,  $c_{i,j}=-c_{j,i}$ for all $i,j\in\Z$.
From this and inserting \eqref{myb4.9} and \eqref{my4.11} into \eqref{my4.10} with the case $k=0$, we get
$c_{i,j}=0$ for all $i,j\in\Z$, which in turn gives $b_{i,j}=0$ by \eqref{my4.10} for any $i$, $j\in\Z$.
That is,  $$\phi_{\lambda}(Y_{i}, Y_{j})=\phi_{\lambda}(M_{i}, L_{j})=0,\quad\forall\ i,j\in\Z.$$

It follows from applying the Jacobi identity to  $(L_{i},M_{j},Y_{k})$,
we have
\begin{equation}\label{pmy4.12}
\mu\phi_{\lambda+\mu}(M_{i+j}, Y_{k})=(\mu+\frac{3}{2}\lambda)\phi_{\mu}(M_{j}, Y_{i+k}).
\end{equation}
Assume  $$\phi_{\lambda}(M_{i},Y_{j})=\sum_{m=0}^{\infty}d_{m}(M_{i},Y_{j})\lambda^{m}$$
for which all but finitely many $d_{m}(M_{i},Y_{j})\in\C$ are nonzero.
Then \eqref{pmy4.12} becomes as
\begin{eqnarray*}
\mu\sum_{m=0}^{\infty}d_{m}(M_{i+j},Y_{k})(\lambda+\mu)^{m}
=(\mu+\frac{3}{2}\lambda)\sum_{m=0}^{\infty}
d_{m}(M_{j},Y_{i+k})\mu^{m}.
\end{eqnarray*}
It is easy to observe that $\phi_{\lambda}(M_{i},Y_{j})=0$. Finally, using  the Jacobi identity to  $(Y_{i},Y_{0},M_{j})$,
one immediately has  $$\phi_{\lambda}(M_{i},M_{j})=0\quad {\rm for\ all}\  i,j\in\Z.$$

From the above discussions and the skew-symmetry we can determine the second cohomology groups of $\mathfrak{csv}$ and $\mathfrak{\tilde{csv}}$.  In fact,   we have obtained the first statement of the following result.
\begin{theo}\label{theo5.1}
{\rm(1)} $H^{2}(\mathfrak{csv}, \C)=\C{\phi},$
 where
$\phi_{\lambda}(L_{i}, L_{j})=a_{i+j}\lambda^3$  all $i,j\in\Z$
$($all other terms vanish$)$ and $a$ is a nonzero complex function.

 {\rm(2)} $H^{2}(\mathfrak{\tilde{csv}}, \C)=\C{\phi}\oplus\C{\varphi}\oplus\C\psi,$
 where $\phi$ is as in (1) and $\varphi,\psi$  are  defined by
 $$\varphi_{\lambda}(L_{i}, N_{j})=-\varphi_{-\lambda}(N_{j}, L_{i})=-e_{i+j}\lambda^2,\quad \psi_\lambda(N_i,N_j)=q_{i+j}\lambda\ {\rm (all\ other\ terms\ vanish)}$$
for  all $i,j\in\Z$, where $e,q$ are two nonzero complex functions.
\end{theo}

\begin{proof}{\rm(2)}\ It suffices to determine $\phi_\lambda(L_i, N_j)$, $\phi_\lambda(N_i,M_j)$, $\phi_\lambda(N_i,Y_j)$ and $\phi_\lambda(N_i,N_j)$.  Applying the Jacobi identity to the triple $(N_{i},L_{j},L_{0})$ and by the similar arguments as for the case $(M_{i},L_{j},L_{0})$ we can deduce that
\begin{equation}\label{nl4.9}
\phi_{\lambda}(N_{i},L_{j})=e_{i,j}\lambda^2
\end{equation}
for some $e_{i,j}\in\C$ and all $i,j\in\Z$.
Applying the Jacobi identity to the triple $(N_{i},L_{j},L_{k})$ and using \eqref{nl4.9}, we have $$e_{i+j,k}=e_{i,j+k}=e_{i+k,j},\quad \forall\ i,j,k\in\Z.$$
So there is a complex function $e$ such that
\begin{equation*}
\phi_{\lambda}(N_{i}, L_{j})=e_{i+j}\lambda^2,\quad\forall\ i,j\in\Z.
\end{equation*}

While from the triple $(N_i,N_j,M_k)$ we see that
$$
\phi_{\lambda}(N_{i}, M_{j+k})=\phi_{\mu}(N_{j}, M_{i+k}),\quad \forall\ i, j, k\in\Z.
$$ Hence, $\phi_\lambda(N_i, M_j)$ is a constant term.

It follows from applying the Jacobi identity to triples $(L_i,N_j,M_k)$, $(L_i,N_j,Y_k)$ and $(L_i,N_j,N_k)$ respectively,   that
$$\phi_\lambda(N_{i},M_{j})= \phi_{\lambda}(N_{i},Y_{j})=0,\  \phi_\lambda(N_i,N_j)=q_{i+j}\lambda,\quad \forall\ i,j\in\Z, $$
where $q$ is a complex function.  This completes the proof of (2).
\end{proof}

\section{Conformal modules of rank one over $\mathfrak{csv}$ and $\mathfrak{\tilde{csv}}$}

A  conformal module $M$  over a Lie conformal algebra $R$ is called a free $R$-module of rank $n$ if $M$ is a {\em free $\C[\partial]$-module of rank $n$}.

First we introduce a class of  nontrivial free conformal modules of rank one over $\mathcal{CHV}$:  Given $a,b,e\in\C$ and $c\in\C^*,$
let $M_{a,b,c,e}=\C[\partial]v$ and
define \begin{eqnarray*}
L_i\,{}_\lambda\, v=c^i(\partial+a\lambda+b)v,\  M_i\,{}_\lambda\, v=c^iev\quad \mbox{for any} \ i\in\Z.
\end{eqnarray*}
It follows from \cite[Theorem 5.3]{FSW} that we have the following lemma.
\begin{lemm}\label{lemma5.1}
Let $M$ be a   nontrivial free conformal module of rank one over $\mathcal{CHV}.$
Then $M$ is isomorphic to $M_{a,b,c,e}$.
\end{lemm}
 Extend the ${\mathcal{CHV}}$-action on $M_{a,b,c,0}$ to  $\mathfrak{csv}$  by defining
 \begin{eqnarray*}
 Y_i\,{}_\lambda\, v=0,\quad\forall\ i\in\Z.
 \end{eqnarray*}In this case we denote $M_{a,b,c,0}$  by $M_{a,b,c,0,0}$. And  $M_{a,b,c,0}$ is denoted by $M_{a,b,c,0,0,d}$ if we furthermore require
\begin{equation*}
N_i\,{}_\lambda\, v=dc^i  v\quad{\rm for\ all}\ i\in\Z\ \mbox{and some}\  c\in\C^*,d\in\C.
\end{equation*}

All nontrivial free  conformal modules of rank one  over the loop Schr\"{o}dinger-Virasoro Lie conformal algebra $\mathfrak{csv}$ and  the extended loop Schr\"{o}dinger-Virasoro Lie conformal algebra
$\mathfrak{\tilde{csv}}$ can be determined, as stated in the following result.

\begin{theo}\label{theo6.1}
{\rm(1)} Suppose that $M$ is a  nontrivial free  conformal module of rank one over $\mathfrak{csv}$.
Then $M$ is isomorphic to $M_{a,b,c,0,0}$ for some  $a,b\in\C$ and $c\in\C^*$.

{\rm(2)} Suppose that $M$ is a  nontrivial free  conformal module of rank one over $\mathfrak{\tilde{csv}}$.
Then $M$ is isomorphic to $M_{a,b,c,0,0,d}$ for some  $a,b,d\in\C$ and $c\in\C^*$.
\end{theo}
\begin{proof}
{\rm(1)}  By Lemma \ref{lemma5.1}, we can assume that
\begin{eqnarray*}
L_i\,{}_\lambda\, v=c^i(\partial+a\lambda+b)v,\  M_i\,{}_\lambda\, v=c^iev,\  Y_i\,{}_\lambda\, v=
h_i(\partial,\lambda)v
\end{eqnarray*}
 for some $h_i(\partial,\lambda)\in \C[\partial,\lambda]$, $a,b,e\in\C,$ $c\in\C^*$ and all $i\in\Z.$
From $[M_i\, {}_\lambda \, Y_j]=0$ one has
\begin{equation*}
h_j(\partial+\lambda,\mu)e=h_j(\partial,\mu)e,\quad \forall\ j\in\Z,
\end{equation*}
 which forces either $e=0$ or
 \begin{equation}\label{eq-h_jjj}
 h_{j}(\mu):=h_{j}(\partial,\mu)\in\C[\mu],\quad \forall\ j\in\Z.
 \end{equation}
From  $[Y_i\, {}_\lambda \, Y_j]=(\partial+2\lambda)M_{i+j}$ one has
\begin{equation*}
\label{g7.6}(\lambda-\mu)ec^{i+j}=h_{j}(\partial+\lambda,\mu)h_{i}(\partial,\lambda)-h_{i}(\partial+\mu,\lambda)h_{j}(\partial,\mu),\quad \forall\ i,j\in\Z.
\end{equation*}
This and \eqref{eq-h_jjj} imply $e=0$, which in turn implies the validity of \eqref{eq-h_jjj}.
 From $[L_i\, {}_\lambda \, Y_j]=(\partial+\frac{3}{2}\lambda) Y_{i+j}$ one has
$(\frac{1}{2}\lambda-\mu)h_{i+j}(\lambda+\mu)=
 -c^i\mu h_{j}(\mu),
$
which implies $ h_{i}(\lambda)=0$ for any $i\in\Z.$
Hence, we get
$
M_i\,{}_\lambda\, v=Y_i\,{}_\lambda\, v=0$ for any  $i\in\Z$.
 This completes the proof of (1).

Finally, note that $$\bigoplus_{i\in\Z}(\C[\partial]L_i\oplus\C[\partial]M_i)\cong \bigoplus_{i\in\Z}(\C[\partial]L_i\oplus\C[\partial]N_i)$$ as Lie conformal algebras, then by Lemma \ref{lemma5.1} again, there exists $d\in\C$ such that
\begin{equation*}
N_i\,{}_\lambda\, v=c^idv,\quad {\rm \forall}\  i\in\Z,
\end{equation*} completing the proof of (2).\end{proof}

\section{$\Z$-graded free intermediate series modules over $\mathfrak{csv}$ and $\mathfrak{\tilde{csv}}$}
The aim of this section is to classify $\Z$-graded free intermediate series modules over  $\mathfrak{csv}$ and $\mathfrak{\tilde{csv}}$, respectively.

Let  $R$ be a {\it $\Z$-graded} Lie  conformal algebra.
A conformal module $V$ over $R$ is  {\it $\Z$-graded} if $V=\oplus_{i\in\Z}V_i$, where each $V_i$ is a $\C[\partial]$-submodule and $R_i\,{}_\lambda\, V_j\subset V_{i+j}[\lambda]$ for any $i,j\in \Z$; if in addition, each $V_i$ is freely generated by one element $v_i\in V_i$ over $\C[\partial]$, then $V$ is called  a {\it $\Z$-graded free intermediate series module}  (see \cite{CK,GXY}).

First we introduce two classes of $\Z$-graded free intermediate series modules over $\mathcal{CHV}$:
Given $a,b,c\in\C$, let $V_{a,b,c}=\oplus_{i\in\Z}\C[\partial]v_i$ and
define
\begin{eqnarray*}
L_i\,{}_\lambda\, v_m\!\!\!&=&\!\!\!(\partial+a\lambda+b)v_{i+m},\\ \ M_i\,{}_\lambda\,  v_m\!\!\!&=&\!\!\!cv_{i+m}\nonumber\quad \mbox{for any} \ i,m\in\Z.
\end{eqnarray*}
Denote $\{0,1\}^\infty$ to be the set of sequences
$(a_i)_{i\in\Z}$  with $a_i\in\{0,1\}$ for any $i\in\Z$; For $A\in\{0,1\}^\infty$ and $b,c\in\C$,
 another class of  nontrivial $\Z$-graded free intermediate series module  is defined on $V_{A,b,c}=\oplus_{i\in\Z}\C[\partial]v_i$, whose
$\lambda$-actions are given by
\begin{eqnarray*}
L_i\,{}_{{}_\lambda} v_m&\!\!\!=\!\!\!&
\begin{cases}
(\partial+b)v_{i+m} &\ \mbox{if}\  (a_m,a_{i+m})=(0,0),\\[4pt]
(\partial+b+\lambda)v_{i+m}&\  \mbox{if} \ (a_m,a_{i+m})=(1,1),\\[4pt]
v_{i+m} &\ \mbox{if} \ (a_m,a_{i+m})=(0,1),\\[4pt]
(\partial+b)(\partial+b+\lambda)v_{i+m}&\  \mbox{if} \ (a_m,a_{i+m})=(1,0)
\end{cases}\\
\label{m7.4}M_i\,{}_\lambda\, v_m&\!\!\!=\!\!\!&cv_{i+m}\nonumber\quad \mbox{for any} \ i,m\in\Z.
\end{eqnarray*}

The following lemma follows from \cite[Theorem 6.6]{FSW} and \cite[Theorem 5.10]{WCY}.
\begin{lemm}\label{lemma-V}
Let $V$ be a nontrivial $\Z$-graded free intermediate series module over $\mathcal{CHV}$.
Then $V$ is isomorphic to $V_{a,b,c}$  or $V_{A,b,c}$.
\end{lemm}

 Extend the $\mathcal{CHV}$-action on $V_{a,b,0}$  and $V_{A,b,0}$ to  $\mathfrak{csv}$  by defining
 \begin{eqnarray*}
 Y_i\,{}_\lambda\, v_m=0,\quad\forall\ i,m\in\Z.
 \end{eqnarray*}In this case we denote $V_{a,b,0}$ and  $V_{A,b,0}$ by $V_{a,b,0,0}$ and $V_{A,b,0,0},$ respectively.  If we furthermore require
 \begin{equation*}
  N_i\,{}_\lambda\, v_m=dv_{i+m}\quad   {\rm for\ all}\ i,m\in\Z\ {\rm and\ some}\   d\in\C,
 \end{equation*}$V_{a,b,0}$ and  $V_{A,b,0}$ are denoted by $V_{a,b,0,0,d}$ and $V_{A,b,0,0,d},$ respectively.

 \begin{theo}\label{theo7.6}
{\rm(1)} Suppose that $V$ is a nontrivial $\Z$-graded free intermediate series module over $\mathfrak{csv}$.
Then $V$ is isomorphic to either $V_{a,b,0,0}$ or $V_{A,b,0,0}$ for some $A\in\{0,1\}^\infty$ and $a,b\in\C$.

{\rm(2)} Suppose that $V$ is a nontrivial $\Z$-graded free intermediate series module over $\mathfrak{\tilde{csv}}$.
Then $V$ is isomorphic to either $V_{a,b,0,0,d}$  or
 $V_{A,b,0,0,d}$ for some $A\in\{0,1\}^\infty$ and $a,b, d\in\C$.
\end{theo}
\begin{proof}
 By Lemma \ref{lemma-V}, assume that
\begin{eqnarray*}&&L_i\,{}_\lambda\, v_m=f_{i,m}(\partial,\lambda)v_{i+m},\ M_i\,{}_\lambda\, v_m=cv_{i+m},\ Y_i\,{}_\lambda\, v_m=h_{i,m}(\partial,\lambda)v_{i+m}\  \end{eqnarray*}for  some  $f_{i,m}(\partial,\lambda), h_{i,m}(\partial,\lambda)\in\C[\partial,\lambda]$ and $c\in\C.$
It follows from $[M_i\,{}_\lambda\, Y_j]=0$ that \begin{equation*}\label{f149}c\cdot h_{j,m}(\partial+\lambda,\mu)=c\cdot h_{j,i+m}(\partial,\mu),\quad \forall\ i,j,m\in\Z,\end{equation*} which forces either $c=0$ or  \begin{equation}\label{eq-h_j}h_{j,m}(\mu):=h_{j,m}(\partial,\mu)\in\C[\mu],\quad \forall\ j,m\in\Z.\end{equation} From  $[Y_i\, {}_\lambda \, Y_j]=(\partial+2\lambda)M_{i+j}$ one has
\begin{equation*}
(\lambda-\mu)c=h_{j,m}(\partial+\lambda,\mu)h_{i,j+m}(\partial,\lambda)-h_{i,m}(\partial+\mu,\lambda)h_{j,i+m}(\partial,\mu),\quad \forall\ i,j,m\in\Z.
\end{equation*}
This and \eqref{eq-h_j} imply $c=0$, which in turn forces the validity of \eqref{eq-h_j} for $j=0$.
Then it follows from this and applying both sides of $[L_i\, {}_\lambda \, Y_0]=(\partial+\frac{3}{2}\lambda) Y_{i}$ on $v_m$ that
\begin{equation}\label{eq-hf}
(\frac{1}{2}\lambda-\mu) h_{i,m}(\partial,\lambda+\mu)\!=\!h_{0,m}(\mu)f_{i,m}(\partial,\lambda)-h_{0,i+m}(\mu)f_{i,m}(\partial+\mu,\lambda).
\end{equation}
Setting $\mu=i=0$ in \eqref{eq-hf}, we get $h_{0,m}(\lambda)=0$ for any $m\in\Z.$ Then setting $\mu=0$ in \eqref{eq-hf} again, we have $h_{i,m}(\partial,\lambda)=0$ for any $i,m\in\Z.$
This completes the proof of (1).

By the similar reason as for Theorem \ref{theo6.1} (2), there exists some $d\in\C$ such that
\begin{equation*}
N_i\,{}_\lambda\, v_m=dv_{i+m},\quad \forall \ i\in\Z,
\end{equation*}
hence completing the proof (2).
\end{proof}

\end{document}